\documentclass[oneside,english]{amsart}
\usepackage[T1]{fontenc}
\usepackage[utf8]{inputenc}
\usepackage{babel}
\usepackage{amsthm}
\usepackage{amssymb}
\usepackage{xargs}[2008/03/08]
\usepackage[unicode=true,pdfusetitle,
 bookmarks=true,bookmarksnumbered=false,bookmarksopen=false,
 breaklinks=false,pdfborder={0 0 1},backref=false,colorlinks=false]
 {hyperref}

\makeatletter
\numberwithin{equation}{section}
\numberwithin{figure}{section}
\theoremstyle{plain}
\newtheorem{thm}{\protect\theoremname}
  \theoremstyle{plain}
  \newtheorem{lem}[thm]{\protect\lemmaname}
  \theoremstyle{remark}
  \newtheorem{rem}[thm]{\protect\remarkname}
  \theoremstyle{definition}
  \newtheorem{defn}[thm]{\protect\definitionname}
  \theoremstyle{plain}
  \newtheorem{prop}[thm]{\protect\propositionname}

\usepackage{CJKutf8}
\usepackage[overlap, CJK]{ruby}
\newenvironment{SChinese}{%
\CJKfamily{gbsn}%
\CJKtilde
\CJKnospace}{}

\makeatother

  \providecommand{\definitionname}{Definition}
  \providecommand{\lemmaname}{Lemma}
  \providecommand{\propositionname}{Proposition}
  \providecommand{\remarkname}{Remark}
\providecommand{\theoremname}{Theorem}

\begin{document}

\title{$p$-adic quantum hyperenveloping algebras for $\mathfrak{sl}_{2}$.}
\begin{abstract}
We construct an example of quantum hyperenveloping algebra over discretely
valued field for the Lie algebra $\mathfrak{sl}_{2}$. 
\end{abstract}

\author{Anton Lyubinin}

\email{anton@ustc.edu.cn, anton@lyubinin.kiev.ua}

\address{Department of Mathematics, School of Mathematical Sciences, \begin{CJK}{UTF8}{}\begin{SChinese}
中国科学技术大学
\end{SChinese}
\end{CJK}University of Science and Technology of China, Hefei, Anhui, People's
Republic of China}

\maketitle
\tableofcontents{}

\global\long\def\emb{\hookrightarrow}
\global\long\def\from{\leftarrow}
\global\long\def\lfr{\leftarrow}
\global\long\def\cten{\widehat{\otimes}}
\global\long\def\uq#1{U_{q}\left(\mathfrak{#1}\right)}
\global\long\def\mf#1{\mathfrak{#1}}
\global\long\def\uqk#1#2{U_{q}\left(\mathfrak{#1}_{#2}\right)}
\global\long\def\norm#1{\left\Vert #1\right\Vert }
\newcommandx\uqr[3][usedefault, addprefix=\global, 1=L]{U_{q}\left(\mathfrak{#2}_{#1},\,#3\right)}
\global\long\def\huq#1#2{\widehat{U}_{q}\left(\mathfrak{#1}_{#2}\right)}

The study of quantum groups and deformations in $p$-adic setting
was proposed in \cite{S}. In this paper we work out the basic example
of quantum hyperenveloping algebra $\huq{sl}{2,L}$ over discretely
valued field $L$ for the Lie algebra $\mathfrak{sl}_{2}$. Our construction
is slightly different from the one in \cite{S} and the use of quantum
doubles allows us to simplify certain proofs. The general case of
finite-dimensional semisimple Lie algebra $\mathfrak{g}$ will be
treated in \cite{L}.

In our construction for $\mathfrak{sl}_{2}$ we use algebras of iterated
skew-commutative power series, which, unlike their polynomial and
formal power series counterparts, seem not to have been studied in
the literature. They possess properties, similar to the ones of Tate
algebras and in our treatment of them we mostly follow \cite{FVDP}.
We believe that these algebras can be of some interest in addition
to this work.

Throughout this paper the base field is denoted by $L$ and a norm
$\norm{\cdot}_{V}$ on an $L$-vector space $V$ is always supposed
to be solid (i.e. $\norm V_{V}\subset\left|L\right|_{L}$) and non-archimedean
(i.e. satisfying strong triangle inequality).

\section{Preliminaries.}

Here we will recall some notions, mostly from the theory of quantum
groups. There are numerous references on that subject, we will use
\cite{KSch}. For any unknown notation in this section one should
look in {[}loc.cit.{]}.

\subsection{\label{sub:.uq}$\uqk{sl}{2,L}$.}

The quantum enveloping algebra $\uqk{sl}{2,L}$ \cite[3.1.1]{KSch}
is the associative algebra over the field $L\left(q\right)$ with
generators $E,\mbox{ }F,\mbox{ }K\mbox{ and }K^{-1}$ and the following
relations 
\begin{equation}
\begin{array}{c}
K\cdot K^{-1}=K^{-1}\cdot K=1\mbox{ ,}\\
KE=q^{2}EK\mbox{ ,}\\
KF=q^{-2}FK\mbox{ ,}\\
EF-FE={\displaystyle \frac{K-K^{-1}}{q-q^{-1}}}\mbox{ .}
\end{array}\label{eq:sl2relations}
\end{equation}

The set $\left\{ F^{l}K^{m}E^{n}\right\} $, $n,l>0$ and $m\in\mathbb{Z}$,
is a basis for $\uqk{sl}{2,L}$.

The Hopf algebra structure on $U_{q}\left(\mathfrak{sl}_{2,L}\right)$
is given by comultiplication
\begin{equation}
\begin{array}{c}
\Delta\left(E\right)=E\otimes K+1\otimes E\\
\Delta\left(F\right)=F\otimes1+K^{-1}\otimes F\\
\Delta\left(K\right)=K\otimes K
\end{array},\label{eq:sl2comultiplication}
\end{equation}
 counit 
\begin{equation}
\begin{array}{ccc}
\epsilon\left(K\right)=1, & \epsilon\left(F\right)=0, & \epsilon\left(E\right)=0,\end{array}\label{eq:sl2counit}
\end{equation}
and antipode 
\begin{equation}
\begin{array}{ccc}
S\left(F\right)=-KF, & S\left(E\right)=-EK^{-1}, & S\left(K\right)=K^{-1}\end{array}.\label{eq:sl2antipode}
\end{equation}
$U_{q}\left(\mathfrak{sl}_{2,L}\right)$ is neither commutative nor
cocommutative. We denote by $\uqk hL$, $\uqk b{L,-}$ and $\uqk b{L,+}$
the subalgebras of $\uqk{sl}{2,L}$ generated by $\left\{ K^{\pm1}\right\} $,
$\left\{ K^{\pm1},F\right\} $ and $\left\{ K^{\pm1},E\right\} $
respectively. One can see that $\uqk hL$, $\uqk b{L,-}$ and $\uqk b{L,+}$
are Hopf subalgebras.

\subsection{Quantum doubles}

We also recall the definition of the quantum double $D\left(A,B,\sigma\right)$
\cite[8.2.1]{KSch} of two skew-paired bialgebras $A$ and $B$. Recall
that $\sigma:A\times B\to L$ is a skew-pairing if $\sigma\left(\cdot,\cdot\right)$
is a dual pairing of bialgebras $A$ and $B^{op}$. We say that $\sigma$
is (convolution) invertible if there exists $\bar{\sigma}:A\times B\to L$
s.t. $\sigma\star\bar{\sigma}=\bar{\sigma}\star\sigma=\epsilon_{A}\otimes\epsilon_{B}$,
i.e. 
\[
\sigma\left(a_{\left(1\right)},b_{\left(1\right)}\right)\bar{\sigma}\left(a_{\left(2\right)},b_{\left(2\right)}\right)=\bar{\sigma}\left(a_{\left(1\right)},b_{\left(1\right)}\right)\sigma\left(a_{\left(2\right)},b_{\left(2\right)}\right)=\epsilon_{A}\left(a\right)\epsilon_{B}\left(b\right).
\]
For Hopf algebras $A$ (or $B$) is a Hopf algebra (with invertible
antipode) than any $\sigma$ is invertible with $\bar{\sigma}\left(a,b\right)=\sigma\left(S_{A}\left(a\right),b\right)$
(resp. $\bar{\sigma}\left(a,b\right)=\sigma\left(a,S_{B}^{-1}\left(b\right)\right)$)
for $a\in A$, $b\in B$.

For bialgebras $A$ and $B$ with an invertible skew-pairing $\sigma$
the quantum double $D\left(A,B,\sigma\right)$ (or simply $D\left(A,B\right)$
when the choice of $\sigma$ is clear) is defined as the $L$-algebra
$B\otimes A$ with the new multiplication 
\begin{equation}
\left(b\otimes a\right)\left(b'\otimes a'\right)=\sum bb_{\left(2\right)}'\bar{\sigma}\left(a_{\left(1\right)},b_{\left(1\right)}'\right)\otimes a'a_{\left(2\right)}\sigma\left(a_{\left(3\right)},b_{\left(3\right)}'\right)\label{eq:double_product}
\end{equation}
 for $a,a'\in A$ and $b,b'\in B$. With tensor product coalgebra
structure of $B\otimes A$, $D\left(A,B,\sigma\right)$ is a bialgebra.
If $A$ and $B$ are Hopf algebras, then $D\left(A,B,\sigma\right)$
is also a Hopf algebra with the antipode 
\[
S\left(b\otimes a\right)=\left(1\otimes S\left(a\right)\right)\left(S\left(b\right)\otimes1\right).
\]

Similarly one can define the double $D\left(A,B,\sigma\right)$ for
two Banach or topological bialgebras (Hopf algebras), using topological
(complete) tensor product and assuming that $\sigma$ is continuous.

One has a unique pairing of Hopf algebras $\left\langle \cdot,\cdot\right\rangle :\uqk b{L,+}\times\uqk b{L,-}^{op}\to L$
\cite[6.3.1]{KSch} s.t. 
\begin{equation}
\left\langle K,K\right\rangle =q^{-2},\;\left\langle K,F\right\rangle =\left\langle E,K\right\rangle =0,\;\left\langle E,F\right\rangle =\left(q^{-1}-q\right)^{-1}.\label{eq:sl2bpbm_pairing}
\end{equation}
 The quantum double $D\left(\uqk b{L,+},\uqk b{L,-}\right):=D\left(\uqk b{L,+},\uqk b{L,-},\left\langle \cdot,\cdot\right\rangle \right)$
\cite[8.2.4]{KSch} is a Hopf algebra, which as an algebra is generated
by elements $E$, $K^{\pm1}$, $F$ and $K_{-}^{\pm1}$ with defining
relations of $\uqk b{L,+}$ and $\uqk b{L,-}$ and cross-relations
\begin{equation}
\begin{array}{c}
KK_{-}=K_{-}K,\:{\displaystyle EF-FE=\frac{K-K_{-}^{-1}}{q-q^{-1}}},\\
K_{-}EK_{-}^{-1}=q^{2}E,\: KFK^{-1}=q^{-2}F.
\end{array}\label{eq:double_relations}
\end{equation}
 The two-sided ideal $I_{q}$ of $D\left(\uqk b{L,+},\uqk b{L,-}\right)$,
generated by $K-K_{-}$, is a Hopf ideal and we have canonical isomorphism
of Hopf algebras 
\[
D\left(\uqk b{L,+},\uqk b{L,-}\right)/I_{q}\cong\uqk{sl}{2,L}.
\]

\subsection{\emph{Normed algebras.}}

Definitions of a normed algebra varies slightly in literature.

We will say that $A$ is a normed algebra over $L$, if $A$ is a
normed $L$-vector space with continuous multiplication. In other
words, the product on $A$ satisfy the inequality $\norm{ab}_{A}\leq C\norm a_{A}\norm b_{A}$
for all $a,b\in A$ and some $C\in\mathbb{R}$. We will also say that
$\norm{\cdot}_{A}$ is submultiplicative if we have $\norm{ab}_{A}\leq\norm a_{A}\norm b_{A}$
and multiplicative if $\norm{ab}_{A}=\norm a_{A}\norm b_{A}$. If
$A$ is complete we will say that $A$ is an $L$-Banach algebra.

Define $A^{0}=\left\{ x\in A|\left\Vert x\right\Vert _{A}\leq1\right\} $,
$A^{00}=\left\{ x\in A|\left\Vert x\right\Vert _{A}<1\right\} $ and
$\bar{A}=A^{0}/A^{00}$. $A^{0}$ and $A^{00}$ are closed $L^{0}-$submodules
of $A$ and $\bar{A}$ is a vector space over the residue field $\bar{L}=L^{0}/L^{00}$.
Denote the image of $f\in A^{0}$ in $\bar{A}$ by $\bar{f}.$

If the norm $\norm{\cdot}_{A}$ is submultiplicative, $A$ is a filtered
ring with the filtration 
\[
F_{r}A=\left\{ a\in A|\norm a_{A}\leq r\right\} 
\]
induced by norm. Since our norms are solid, this filtration is quasi-integral.
If $A$ is an $L$-Banach algebra, this filtration is complete. The
associated graded ring $Gr_{\cdot}A$ is a $Gr_{\cdot}L$-algebra,
with $Gr_{\cdot}L=\bar{L}[\bar{\pi},\bar{\pi}^{-1}]$ is the ring
of Laurent polynomials, where $\pi$ is a uniformizer of $L$. $Gr_{\cdot}A$
does not have zero-divisors iff $\norm{\cdot}_{A}$ is multiplicative.
If $Gr_{\cdot}A$ is left (right) noetherian then $A$ is left (right)
noetherian \cite[1.1]{ST33}.

\section{Skew-Tate algebras.}

Let $L$ be a discretely valued field and $A$ be a $L$-Banach algebra. 
\begin{lem}
\label{lem:Let-A-be} Let $B=A\left[x,\alpha,\delta\right]$ be an
Ore extension of A ($\alpha:A\to A$ automorphism of A and $\delta$
is an $\alpha-$derivation of A). Consider a ``Gauss R-norm'' on
B: $f\in B,$ $f=\sum f_{n}x^{n}$, then $\norm f_{R}\equiv\left\Vert f\right\Vert _{Gauss,R}=\max_{n}\left\Vert f_{n}\right\Vert _{A}R^{n}$
($R\in\left|L\right|_{L}\subset\mathbb{R}$). Suppose $\left\Vert \alpha\right\Vert \leq1,$
$\left\Vert \delta\right\Vert \leq1$ and $\left|R\right|\geq1.$
Then $\norm{\cdot}_{R}$ is a submultiplicative non-archimedean algebra
norm on B.\end{lem}
\begin{proof}
It is clear that $\norm{\cdot}_{R}$ is a $L-$vector space norm (as
in commutative case).

Lets prove submultiplicativity, i.e. $\left\Vert fg\right\Vert _{R}\leq\left\Vert f\right\Vert _{R}\left\Vert g\right\Vert _{R}.$

Recall that in $B$ we have $x\cdot a=\alpha\left(a\right)x+\delta\left(a\right),$
$a\in A.$ One can prove (using induction) that $x^{n}\cdot a=\sum_{k=0}^{n}c_{nk}\left(a\right)x^{k},$
where $c_{nk}\left(a\right)$ is the sum of all words with $k-$letters
$\alpha$ and $\left(n-k\right)-$letters $\delta,$ applied to $a.$
Since $\left\Vert \alpha\right\Vert \leq1,$ $\left\Vert \delta\right\Vert \leq1$,
then $\left\Vert c_{nk}\left(a\right)\right\Vert _{A}\leq\left\Vert a\right\Vert _{A}.$
Now let $f=\sum_{n=0}^{l}f_{n}x^{n}$ and $g=\sum_{k=0}^{s}g_{k}x^{k}.$
Then 
\[
\left\Vert fg\right\Vert _{R}=\left\Vert \left(\sum_{n=0}^{l}f_{n}x^{n}\right)\left(\sum_{k=0}^{s}g_{k}x^{k}\right)\right\Vert _{R}=\left\Vert \sum_{n=0}^{l}f_{n}\left(\sum_{k=0}^{s}\left(x^{n}g_{k}\right)x^{k}\right)\right\Vert _{R}=
\]
\[
=\left\Vert \sum_{n=0}^{l}\sum_{k=0}^{s}\left(f_{n}\sum_{i=0}^{n}c_{ni}\left(g_{k}\right)x^{i}\right)x^{k}\right\Vert _{R}=\left\Vert \sum_{n=0}^{l}f_{n}\sum_{i=0}^{n}\left(\sum_{k=0}^{s}c_{ni}\left(g_{k}\right)x^{k+i}\right)\right\Vert _{R}\leq
\]
\[
\leq\max_{n}\left(\left\Vert \sum_{i=0}^{n}\sum_{k=0}^{s}f_{n}c_{ni}\left(g_{k}\right)x^{k+i}\right\Vert _{R}\right)=\max_{n}\max_{i}\max_{k}\left(\left\Vert f_{n}\right\Vert _{A}\left\Vert c_{ni}\left(g_{k}\right)\right\Vert _{A}R^{k+i}\right)\leq
\]
\[
=\max_{n}\max_{i}\max_{k}\left(\left\Vert f_{n}\right\Vert _{A}R^{i}\right)\left(\left\Vert g_{k}\right\Vert _{A}R^{k}\right)\leq\left\Vert g\right\Vert _{R}\cdot\max_{n,i}\left\Vert f_{n}\right\Vert _{A}R^{i}=
\]
\[
=\left\Vert g\right\Vert _{R}\cdot\max_{n}\left\Vert f_{n}\right\Vert _{A}R^{n}=\left\Vert f\right\Vert _{R}\left\Vert g\right\Vert _{R}.
\]

\end{proof}
Denote the completion of $B=A\left[x,\alpha,\delta\right]$ w.r.t.
$\left\Vert \cdot\right\Vert _{R}$ by $A\left\{ x/R,\alpha,\delta\right\} .$ 
\begin{rem}
\label{rem:skew-tate_isom}Clear that if $\left|s\right|_{L}=R$ for
some $s\in L$, then $A\left\{ x/R,\alpha,\delta\right\} \cong A\left\{ z,\alpha,s^{-1}\delta\right\} $
with $x$ mapped to $sz.$\end{rem}
\begin{defn}
An algebra of the form $L\left\{ x_{1},\alpha_{1},\delta_{1}\right\} \dots\left\{ x_{n},\alpha_{n},\delta_{n}\right\} $
with $\left\Vert \cdot\right\Vert _{Gauss}=\left\Vert \cdot\right\Vert _{1}$
will be called skew-Tate algebra.

In order to check whether skew-Tate algebras have zero-divisors, we
will study the algebras $\bar{B}$ for $B=A\left\{ x,\alpha,\delta\right\} $.
For $\norm{\delta}\leq1$ the map $\bar{\delta}:\bar{A}\to\bar{A}$,
$\bar{\delta}\left(\bar{a}\right):=\overline{\delta\left(a\right)}$
is a well-defined derivation and if $\alpha:A\to A$ is an isometry,
then $\bar{\alpha}:\bar{A}\to\bar{A}$ is an automorphism. Thus the
Ore extension $\bar{A}\left[\bar{x},\bar{\alpha},\bar{\delta}\right]$
is well-defined and is equal to $\bar{B}$.\end{defn}
\begin{prop}
Let $\alpha$ be an isometry. Then $\norm{\cdot}_{B}$ is multiplicative
if $\norm{\cdot}_{A}$ is multiplicative. In particular, for isometric
$\alpha_{n}$ skew-Tate algebras are (left and right) noetherian with
multiplicative norms.\end{prop}
\begin{proof}
Follows from the fact that $\bar{A}\left[\bar{x},\bar{\alpha},\bar{\delta}\right]$
does not have zero divisors if $\bar{A}$ does not have ones. Second
statement is true since skew-polynomial rings are (left and right)
noetherian.\end{proof}
\begin{rem}
For $\delta=0$, $\alpha$ an isometry and $0<R<1$, the Gauss norm
$\norm{\cdot}_{R}$ is a multiplicative algebra norm if $\norm{\cdot}_{A}$
is multiplicative. The proof is the same as in commutative case.
\end{rem}
For Tate algebra we have Weierstrass division and preparation theorems.
One has similar results for skew-Tate algebras.
\begin{defn}
An element $f\in A\left\{ x,\alpha,\delta\right\} $ with $\left\Vert f\right\Vert _{Gauss}=1$
is called regular of degree d if $\bar{f}$ has the form $\lambda z^{d}+{\displaystyle \sum_{i=0}^{d-1}c_{i}z^{i}}$
with $\lambda\in\bar{L}^{*}$ and $c_{i}\in\bar{A}.$\end{defn}
\begin{thm}
(Weierstrass division and preparation)
\begin{enumerate}
\item (Division) Let f be a regular element of $A\left\{ z,\alpha,\delta\right\} $
of degree d. Then for any g in $A\left\{ z,\alpha,\delta\right\} $
there exists unique q and r such that g=qf+r and degree of r is less
then d. Moreover, $\left\Vert g\right\Vert _{Gauss}=\max\left(\left\Vert q\right\Vert _{Gauss},\left\Vert r\right\Vert _{Gauss}\right).$
\item (Preparation) Let f be a regular element of $A\left\{ x,\alpha,\delta\right\} $
of degree d. Then there exists $w\in A\left[x,\alpha,\delta\right],$
s.t. $f=w\cdot e,$ where e is a unit in $A\left\{ x,\alpha,\delta\right\} $,
and w is regular of degree d. If $f\in A\left[x,\alpha,\delta\right]$
then also $e\in A\left[x,\alpha,\delta\right]$.
\end{enumerate}
\end{thm}
\begin{proof}
(1) Take $f=f_{0}-D,$ where $f_{0}=\lambda z^{d}+{\displaystyle \sum_{i=0}^{d-1}c_{i}z^{i}}$,
$c_{i}=A^{0}$ and $\left\Vert D\right\Vert _{Gauss}<1.$ Lets prove
that the statement of (1) is true for $f_{0}.$ Lets first prove the
statement for powers of $z,$ i.e. $z^{i}=q_{i}f_{0}+r_{i}.$ We have
the identity 
\[
z^{d}=\lambda^{-1}\left(\lambda z^{d}+{\displaystyle \sum_{i=0}^{d-1}c_{i}z^{i}}-{\displaystyle \sum_{i=0}^{d-1}c_{i}z^{i}}\right)=q_{d}f_{0}+r_{d}
\]
 with $q_{d}=\lambda^{-1}$ and $r_{d}=\lambda^{-1}\left({\displaystyle \sum_{i=0}^{d-1}c_{i}z^{i}}\right)$
and for $i<d$ $q_{i}=0.$ Now for $z^{n+1}$ we have $z^{n+1}=z\cdot z^{n}=z\cdot\left(q_{n}f_{0}+r_{n}\right)=\left(z\cdot q_{n}\right)f_{0}+z\cdot r_{n}.$
If $r_{n}={\displaystyle \sum_{i=0}^{d-1}c_{ni}z^{i}}$ then from
commuting relations we compute 
\[
z\cdot r_{n}={\displaystyle \sum_{i=0}^{d-1}\left(z\cdot c_{ni}\right)z^{i}={\displaystyle \sum_{i=0}^{d-1}\left(\alpha\left(c_{ni}\right)z+\delta\left(c_{ni}\right)\right)z^{i}=}}
\]
\[
=\alpha\left(c_{n\left(d-1\right)}\right)z^{d}+{\displaystyle \sum_{i=0}^{d-1}\left(\alpha\left(c_{n\left(i-1\right)}\right)+\delta\left(c_{ni}\right)\right)z^{i}}=
\]
\[
=\alpha\left(c_{n\left(d-1\right)}\right)\left(\lambda^{-1}f_{0}+r_{d}\right)+{\displaystyle \sum_{i=0}^{d-1}\left(\alpha\left(c_{n\left(i-1\right)}\right)+\delta\left(c_{ni}\right)\right)z^{i}}=
\]
\[
=\alpha\left(c_{n\left(d-1\right)}\right)\lambda^{-1}f_{0}+\alpha\left(c_{n\left(d-1\right)}\right)r_{d}+{\displaystyle \sum_{i=0}^{d-1}\left(\alpha\left(c_{n\left(i-1\right)}\right)+\delta\left(c_{ni}\right)\right)z^{i}}.
\]
Thus $z^{n+1}=q_{n+1}f_{0}+r_{n+1},$ where $q_{n+1}=z\cdot q_{n}+\lambda^{-1}\alpha\left(c_{n\left(d-1\right)}\right)$
and $r_{n+1}=\alpha\left(c_{n\left(d-1\right)}\right)r_{d}+{\displaystyle \sum_{i=0}^{d-1}\left(\alpha\left(c_{n\left(i-1\right)}\right)+\delta\left(c_{ni}\right)\right)z^{i}}.$
It is clear that is this formulas the norms of the coefficients do
not increase, and thus for any $g={\displaystyle \sum_{n=0}^{\infty}g_{n}z^{n}}$
we get the equality 
\[
g=\left(\sum_{n=0}^{\infty}g_{n}q_{n}\right)f_{0}+\left(\sum_{n=0}^{\infty}g_{n}r_{n}\right)
\]
with both sums being convergent in $A\left\{ z,\alpha,\delta\right\} $.

Now lets prove the division property for $f.$ We have $f_{0}=f+D$
and for any $g\in A\left\{ z,\alpha,\delta\right\} $ we have the
decomposition 
\[
g=q_{0}f_{0}+r_{0}=q_{0}f+q_{0}D+r_{0}=q_{0}f+g_{1}+r_{0}
\]
 where $g_{1}=q_{0}D.$ Since the norm is submultiplicative, we have
\[
\left\Vert g_{1}\right\Vert _{Gauss}\leq\left\Vert q_{0}\right\Vert _{Gauss}\left\Vert D\right\Vert _{Gauss}\leq\left\Vert g\right\Vert _{Gauss}\left\Vert D\right\Vert _{Gauss}.
\]
We have the same decomposition for $g_{1}$, 
\[
g_{1}=q_{1}f_{0}+r_{1}=q_{1}f+g_{2}+r_{1}
\]
 where $g_{2}=q_{1}D$ and 
\[
\left\Vert g_{2}\right\Vert _{Gauss}\leq\left\Vert q_{1}\right\Vert _{Gauss}\left\Vert D\right\Vert _{Gauss}\leq\left\Vert g_{1}\right\Vert _{Gauss}\left\Vert D\right\Vert _{Gauss}\leq\left\Vert g\right\Vert _{Gauss}\left\Vert D\right\Vert _{Gauss}^{2}.
\]
Continuing by induction, we construct zero sequences $g_{n},$ $q_{n}$
and $r_{n}$ s.t. $g_{n}=q_{n}f+g_{n+1}+r_{n}.$ Adding up all this
recurrent relations gives 
\[
g=\left(\sum_{n=0}^{\infty}q_{n}\right)f+\left(\sum_{n=0}^{\infty}r_{n}\right)
\]
and 
\[
\left\Vert g\right\Vert _{Gauss}=\max\left(\left\Vert \sum_{n=0}^{\infty}q_{n}\right\Vert _{Gauss},\left\Vert \sum_{n=0}^{\infty}r_{n}\right\Vert _{Gauss}\right).
\]

Now lets prove uniqueness. An equality $g=q_{1}f+r_{1}=q_{2}f+r_{2},$
imply $0=\left(q_{1}-q_{2}\right)f+\left(r_{1}-r_{2}\right).$ Since
norm of $f$ is one, we have $\left\Vert q_{1}-q_{2}\right\Vert _{Gauss}=\left\Vert r_{1}-r_{2}\right\Vert _{Gauss}$
and multiplication by an appropriate number makes both norms equal
1. But then in $A\left\{ z,\alpha,\delta\right\} ^{0}/A\left\{ z,\alpha,\delta\right\} ^{00}$
we have $\overline{q_{1}-q_{2}}\cdot\bar{f}=\overline{r_{1}-r_{2}},$
and this is impossible, since on left hand side we have a skew-polynomial
of degree $\geq d$ and on the right hand side $<d$.

(2) Since $f$ is distinguished, by (1) there exists $e'$ and $r'$
s.t. $x^{d}=e'f+r$ and $\mbox{deg}\left(r\right)<d.$ Define $\omega=x^{d}-r.$
We have $\omega=e'f.$ Since $\left\Vert r\right\Vert _{Gauss}\leq\left\Vert x^{d}\right\Vert _{Gauss}=1,$
we get $\left\Vert \omega\right\Vert _{Gauss}=1$ and $\omega$ is
distinguished of degree $d$. Then in $A\left\{ z,\alpha,\delta\right\} ^{0}/A\left\{ z,\alpha,\delta\right\} ^{00}$
we have $\overline{\omega}=\overline{e'}\overline{f}$ with $\bar{\omega}$
and $\bar{f}$ being unitary skew-polynomials of the same degree.
This means that $\bar{e'}$ is a unit in $A\left\{ z,\alpha,\delta\right\} ^{0}/A\left\{ z,\alpha,\delta\right\} ^{00}$
and $e'$ is a unit in $A\left\{ z,\alpha,\delta\right\} .$ If $f$
is a polynomial then also must be $e.$
\end{proof}

\section{Quantum hyperenveloping algebra of $U_{q}\left(\mathfrak{sl}_{2,L}\right)$ }

From now on $q\in L$ is a nonzero number.

\subsection{Norm completions}

We want to define a completion of $U_{q}\left(\mathfrak{sl}_{2,L}\right)$
with respect to an analogue of the Gauss $R$-norm.

In order to do so, lets recall \cite[Prop. 6.1.4]{Kas} that $U_{q}\left(\mathfrak{sl}_{2,L}\right)$
is a noetherian algebra, obtained by a sequence of Ore extensions

\[
\begin{array}{c}
L\left[K,K^{-1}\right]=\uqk hL\emb\uqk b{L,-}=\uqk hL\left[F,\alpha_{0},0\right]\\
\alpha_{0}\left(K\right)=q^{2}K
\end{array}
\]
and 
\begin{equation}
\begin{array}{c}
\uqk b{L,-}\emb\uqk gL=\uqk b{L,-}\left[E,\alpha_{1},\delta\right]\\
\alpha_{1}\left(F^{j}K^{l}\right)=q^{-2l}F^{j}K^{l}\\
\delta\left(F\right)=\frac{K-K^{-1}}{q-q^{-1}}\\
\delta\left(F^{j}K^{l}\right)={\displaystyle \sum_{i=0}^{j-1}F^{j-1}\delta\left(F\right)\left(q^{-2i}K\right)K^{l}}\\
\delta\left(K\right)=0
\end{array}.\label{eq:1}
\end{equation}
Let $\uqr h{R_{K}}$ be the algebra of Laurent series in $K,$ 
\[
\uqr h{R_{K}}:=\left\{ \sum_{n\in Z}f_{n}K^{n}|\mbox{ }\lim_{n\to\pm\infty}\left|f_{n}\right|_{L}R_{K}^{n}=0\right\} 
\]
with fixed $R_{K}.$ It is a Banach $K-$algebra w.r.t. the norm 
\[
\norm f_{R_{K}}=\max_{n\in Z}\left|f_{n}\right|_{L}R_{K}^{n}.
\]
 Let $\left|q\right|_{L}=1.$ Then the map 
\[
\alpha_{0}:\uqr h{R_{K}}\to\uqr h{R_{K}},\mbox{ }\alpha_{0}\left(K\right)=q^{2}K
\]
 is an isometry. By lemma \ref{lem:Let-A-be}, the algebra 
\[
\uqr[L,-]b{R_{K},\, R_{F}}:=\uqr h{R_{K}}\left\{ \frac{F}{R_{F}},\alpha_{0},0\right\} 
\]
 is a Banach $\uqr h{R_{K}}$-algebra, which can be described as 
\[
\uqr[L,-]b{R_{K},\, R_{F}}=\left\{ \sum_{n=0}^{\infty}a_{n}F^{n}\Big|\; a_{n}\in\uqr h{R_{K}},\mbox{ s.t. }\lim_{n\to\infty}\norm{a_{n}}_{R_{K}}R_{F}^{n}=0,\right\} 
\]
or a $K-$Banach algebra of convergent series in $F,\mbox{ }K^{\pm1}$
with radius at least $\left(R_{F},R_{K}\right).$ Similarly one construct
the algebra $\uqr[L,+]b{R_{K},\, R_{E}}$. 

From formulas \ref{eq:sl2comultiplication}, \ref{eq:sl2counit} and
\ref{eq:sl2antipode} one can see that for $R_{K}=1$ the comultiplication,
counit and antipode of $\uqk{sl}{2,L}$ are bounded maps (with $\norm{\Delta},\norm{\epsilon},\norm S\leq1$)
and thus they make the algebras $\uqr[L,-]b{R_{F}}:=\uqr[L,-]b{1,\, R_{F}}$
and $\uqr[L,+]b{R_{E}}:=\uqr[L,+]b{1,\, R_{E}}$ into Banach Hopf
algebras. Also from formulas \ref{eq:sl2bpbm_pairing} one can see
that the pairing $\left\langle \cdot,\cdot\right\rangle :\uqk b{L,+}\times\uqk b{L,-}^{op}\to L$
satisfy the condition $\left|\left\langle x,y\right\rangle \right|_{L}\leq\norm x_{R}\norm y_{R}$
when $R>\left|\left(q^{-1}-q\right)^{-1}\right|_{L}$ and thus this
pairing can be extended to the pairing $\left\langle \cdot,\cdot\right\rangle :\uqr[L,+]bR\times\uqr[L,-]bR^{op}\to L$
of Banach Hopf algebras.

Consider the quantum double $D\left(\uqr[L,+]bR,\uqr[L,-]bR\right)$.
As a Banach space it is equal to $\uqr[L,-]bR\cten\uqr[L,+]bR$. The
remarks in the above paragraph imply that the multiplication \ref{eq:double_product}
on $D\left(\uqr[L,+]bR,\uqr[L,-]bR\right)$ is the composition of
maps of norm $\leq1$, which means that the norm on $D\left(\uqr[L,+]bR,\uqr[L,-]bR\right)$
(the tensor product norm on $\uqr[L,-]bR\cten\uqr[L,+]bR$) is submultiplicative.
Similar to the algebraic case, $D\left(\uqr[L,+]bR,\uqr[L,-]bR\right)$
can be described as the space 
\begin{equation}
\begin{array}{c}
D\left(\uqr[L,+]bR,\uqr[L,-]bR\right)=\\
\left\{ {\displaystyle \sum_{\bar{n}=0}^{\infty}a_{\bar{n}}E^{n_{E}}K^{n_{K}}K_{-}^{n_{K_{-}}}F^{n_{F}}\Big|\;\begin{array}{c}
n_{E},n_{F}\in\mathbb{N},n_{K_{\pm}}\in\mathbb{Z};\: a_{\bar{n}}\in L,\mbox{ s.t. }\\
\lim_{\left|\bar{n}\right|\to\infty}\left|a_{\bar{n}}\right|_{L}R^{n_{E}+n_{F}}=0
\end{array}}\right\} 
\end{array}\label{eq:doubleDescr}
\end{equation}
with multiplication defined by relations of $\uqr[L,-]bR$, $\uqr[L,+]bR$
and \ref{eq:double_relations}.
\begin{lem}
\label{lem:The-double-ass-gr-noeth}The double $D\left(\uqr[L,+]bR,\uqr[L,-]bR\right)$
as an $L$-algebra is a (left and right) noetherian and it's norm
is multiplicative. If $\left|1-q\right|_{L}<1$ then $Gr_{\cdot}D\left(\uqr[L,+]bR,\uqr[L,-]bR\right)$
is commutative.\end{lem}
\begin{proof}
From relations \ref{eq:double_relations} we have an isomorphism 
\[
Gr_{\cdot}D\left(\uqr[L,+]bR,\uqr[L,-]bR\right)\simeq\bar{L}[\bar{\pi}^{\pm1}]\left[\tilde{K}^{\pm1},\tilde{K}_{-}^{\pm1}\right]\left[\tilde{F},\beta_{0},0\right]\left[\tilde{E},\beta_{1},0\right]
\]
and the algebra on the right hand side is (left and right) noetherian
without zero divizors. If $\left|1-q\right|_{L}<1$ then $\bar{q}=\bar{1}$
and the ring 
\[
Gr_{\cdot}D\left(\uqr[L,+]bR,\uqr[L,-]bR\right)\simeq\bar{L}[\bar{\pi}^{\pm1}]\left[\tilde{K}^{\pm1},\tilde{K}_{-}^{\pm1}\right]\left[\tilde{F}\right]\left[\tilde{E}\right]
\]
is commutative.
\end{proof}
We define the Banach Hopf algebra 
\[
\uqr[2,L]{sl}R:=D\left(\uqr[L,+]bR,\uqr[L,-]bR\right)/I_{q},
\]
where $I_{q}$ is the closed Hopf ideal generated by $\left(K-K_{-}\right)$.
As an algebra it is (left and right) noetherian. Multiplicativity
of the norm on $\uqr[2,L]{sl}R$ (the quotient norm from $D\left(\uqr[L,+]bR,\uqr[L,-]bR\right)$)
can be checked similarly to the lemma \ref{lem:The-double-ass-gr-noeth}.

\subsection{quantum hyperenveloping algebra}

Similar to the \ref{eq:doubleDescr}, the Banach Hopf algebras $\uqr[2,L]{sl}R$
can be described as the Banach space of convergent power series 
\[
\left\{ {\displaystyle \sum_{\bar{n}=0}^{\infty}a_{\bar{n}}E^{n_{E}}K^{n_{K}}F^{n_{F}}\Big|\; n_{E},n_{F}\in\mathbb{N},n_{K}\in\mathbb{Z};\: a_{\bar{n}}\in L,\mbox{ s.t. }\lim_{\left|\bar{n}\right|\to\infty}\left|a_{\bar{n}}\right|_{L}R^{n_{E}+n_{F}}=0}\right\} .
\]
From this description it is clear that for $R_{1}<R_{2}$ we have
an injective (and compact) map $\uqr[2,L]{sl}{R_{1}}\from\uqr[2,L]{sl}{R_{2}}$.
Thus we have a projective system of Banach Hopf algebras with injective
and compact transition maps. It's projective limit is a nuclear Fréchet
space, which also have a topological Hopf algebra structure.
\begin{defn}
We define the quantum hyperenveloping algebra of $\mathfrak{sl}_{2,L}$
as $\huq{sl}{2,L}:=\lim_{\from}\uqr[2,L]{sl}R.$

Similarly one can define the Fréchet Hopf algebras $\huq b{-,L}$
and $\huq b{+,L}$. Their double $D\left(\huq b{+,L},\huq b{-,L}\right)$
is a Fréchet Hopf algebra and one has a topological isomorphisms 
\[
D\left(\huq b{+,L},\huq b{-,L}\right)\cong\lim_{\from}D\left(\uqr[+,L]bR,\uqr[-,L]bR\right)
\]
 and 
\[
\huq{sl}{2,L}\cong D\left(\huq b{+,L},\huq b{-,L}\right)/\hat{I_{q}},
\]
where $\hat{I_{q}}$ is the closed ideal, generated by $\left(K-K_{-}\right)$.

Recall that a Fréchet algebra $A$ is called Fréchet-Stein (compare
to \cite[sec. 3]{ST33})if it is a locally convex projective limit
of Banach algebras $A_{n}$ with transition maps $A_{n}\from A_{n+1}$
having dense images, such that\end{defn}
\begin{itemize}
\item each $A_{n}$ is (left) noetherian;
\item the transition maps $A_{n}\from A_{n+1}$ are flat.
\end{itemize}
In order to prove Fréchet-Stein property for $\huq{sl}{2,L}$, we
consider the subalgebra $A\subset D\left(\uqk b{+,L},\uqk b{-,L}\right)$,
generated by $\left\{ K,M=K_{-}^{-1},E,F\right\} $. Clear that under
the quotient map $D\left(\uqk b{+,L},\uqk b{-,L}\right)\to\uqk{sl}{2,L}$
$A$ is mapped onto $\uqk{sl}{2,L}$. $A$ also gives rise to the
Banach subalgebras $A\left(R\right)\subset D\left(\uqr[L,+]bR,\uqr[L,-]bR\right)$
and a Fréchet subalgebra $\widehat{A}\subset D\left(\huq b{+,L},\huq b{-,L}\right)$
with $\widehat{A}\cong\lim_{\from}A\left(R\right)$ and we have a
topological isomorphism $\huq{sl}{2,L}\cong\widehat{A}/\left\langle KM-1\right\rangle $.
Thus in order to show that $\huq{sl}{2,L}$ is Fréchet-Stein, by lemma
\cite[Prop. 3.7]{ST33} it is enough to show that $\widehat{A}$ is
Fréchet-Stein.

Similar to lemma \ref{lem:The-double-ass-gr-noeth} one can show the
algebras $A\left(R\right)$ are (left and right) noetherian with associated
graded rings $Gr_{\cdot}A\left(R\right)\cong\bar{L}\left[\bar{\pi}^{\pm1}\right]\left[\bar{K},\bar{M}\right]\left[\tilde{F},\beta_{0},0\right]\left[\tilde{E},\beta_{1},0\right]$.
Thus we only need to check the second property.
\begin{prop}
Let $A\left(R_{1}\right)\hookrightarrow A\left(R_{2}\right)$ be the
inclusion map for $R_{1}>R_{2}$. Then $A\left(R_{2}\right)$ is a
flat $A\left(R_{1}\right)-$module.\end{prop}
\begin{proof}
The proof follows the idea from \cite[4.8, 4.9]{ST33}. We view our
Banach algebras as complete filtered rings with the filtration induced
by norm. \cite[Prop. 1.2]{ST33} says, that the map between two such
rings is flat if associated graded rings are noetherian and the associated
map of graded rings is flat.

As in \cite[4.9]{ST33} we factor our map $A\left(R_{1}\right)\hookrightarrow A\left(R_{2}\right)$
into $A\left(R_{1}\right)\hookrightarrow A^{\infty}\left(R_{1}\right)\hookrightarrow A\left(R_{2}\right),$
where 
\[
A^{\infty}\left(R_{1}\right)=\left\{ {\displaystyle \sum_{\bar{n}=0}^{\infty}a_{\bar{n}}E^{n_{E}}K^{n_{K}}M^{n_{M}}F^{n_{F}}\Big|\, n_{\cdot}\in\mathbb{N};a_{\bar{n}}\in L:\sup_{\bar{n}}\left|a_{\bar{n}}\right|_{L}R^{n_{E}+n_{F}}\leq\infty}\right\} 
\]
is the algebra with the same relations as for $A\left(R_{1}\right)$.
$A^{\infty}\left(R_{1}\right)$ is a Banach algebra w.r.t. supremum
norm and $A\left(R_{1}\right)$ is a closed subalgebra. Easy to see
that the associated graded ring of $A^{\infty}\left(R_{1}\right)$
is the ring of formal skew-power series $Gr_{\cdot}A^{\infty}\left(R\right)\cong\bar{L}\left[\bar{\pi}^{\pm1}\right]\left[\left[\bar{K},\bar{M}\right]\right]\left[\left[\tilde{F},\beta_{0},0\right]\right]\left[\left[\tilde{E},\beta_{1},0\right]\right].$
Since both rings $Gr_{\cdot}A\left(R\right)$ and $Gr_{\cdot}A^{\infty}\left(R\right)$
are noetherian and inclusion of polynomials into power series is a
flat map, the inclusion $A\left(R_{1}\right)\hookrightarrow A^{\infty}\left(R_{1}\right)$
is flat.

For the second inclusion note that $A^{\infty}\left(R_{1}\right)\cong L\cten F_{0}A^{\infty}\left(R_{1}\right)$
and $A^{\infty}\left(R_{1}\right)\hookrightarrow A\left(R_{2}\right)$
is flat iff $F_{0}A^{\infty}\left(R_{1}\right)\hookrightarrow A\left(R_{2}\right)$
is flat. It follows from compactness of the inclusion map that $F_{0}A^{\infty}\left(R_{1}\right)$
is a closed subset of $A\left(R_{2}\right)$ and thus it is complete
w.r.t. the norm filtration of $A\left(R_{2}\right)$. So one can apply
\cite[Prop. 1.2]{ST33} in this case too. Similar to \cite[Thm. 4.9]{ST33}
one can show that the map of associated graded rings of $F_{0}A^{\infty}\left(R_{1}\right)$
and $A\left(R_{2}\right)$ is a localization and thus is flat. This
proves that the second inclusion is also flat.
\end{proof}
Thus we have proved that $\huq{sl}{2,L}$ is noncommutative and noncocommutative
Fréchet-Stein Hopf algebra. 
\begin{rem}
One can describe $\huq{sl}{2,L}$ as the space 
\[
\left\{ {\displaystyle \sum_{\bar{n}=0}^{\infty}a_{\bar{n}}E^{n_{E}}K^{n_{K}}F^{n_{F}}\Big|\;\begin{array}{c}
n_{E},n_{F}\in\mathbb{N},n_{K_{\pm}}\in\mathbb{Z};\: a_{\bar{n}}\in L,\mbox{ s.t. }\\
\forall R>0:\,\lim_{\left|\bar{n}\right|\to\infty}\left|a_{\bar{n}}\right|_{L}R^{n_{E}+n_{F}}=0
\end{array}}\right\} 
\]
with the locally convex topology given by the system of norms $\nu_{R}$:
\begin{equation}
\nu_{R}:\mbox{ }\nu_{R}\left(\sum_{\bar{n}=0}^{\infty}a_{\bar{n}}E^{n_{E}}K^{n_{K}}F^{n_{F}}\right)=\sup\left(\left|a_{\bar{n}}\right|_{L}R^{n_{E}+n_{F}}\right).\label{eq:norm1}
\end{equation}
 
\end{rem}
Equivalently, one can take the family of norms 
\begin{equation}
\nu_{R}':\mbox{ }\nu_{R}'\left(\sum_{\bar{n}=0}^{\infty}a_{\bar{n}}E^{n_{E}}K^{n_{K}}F^{n_{F}}\right)=\sup\left(\left|a_{\bar{n}}\right|_{L}\left|\left[n_{E}\right]_{q}!\right|_{L}\left|\left[n_{F}\right]_{q}!\right|_{L}R^{n_{E}+n_{F}}\right)\label{eq:norm2}
\end{equation}
 (similar to \cite[1.2.8]{Kohl}). This is possible due to an estimate
(4.1.1.1) from \cite{LDiv}, which implies that $\exists C\geq1:$
\[
\left|\frac{1}{\left[n\right]_{q}!}\right|_{L}\leq C^{n}p^{\frac{n}{p-1}}
\]
(note, that in \cite{LDiv}, $\left[n\right]_{q}=1+q+...+q^{n-1}=\left[\left[n\right]\right]_{q}$
in notations of \cite{KSch}, and one need to use $\left[n\right]_{q}=q^{-n+1}\left[\left[n\right]\right]_{q},$
\cite[6.1.1. (1.7)]{Kas}). The completion of $\huq{sl}{2,L}$ w.r.t.
$\nu_{R}'$ will be denoted by $\uqr[2,L]{sl}{\nu_{R}'}$.

\subsection{Second construction}

Instead of constructing $\uqr[2,L]{sl}R$ through quantum doubles,
one could use skew-Tate algebras one more time. Namely consider $\uqr[-,L]b{R_{K},\, R_{F}}\left[E,\alpha_{1},\delta\right]$,
with $\alpha_{1},\delta$ as in \ref{eq:1}. Since $\left|q\right|_{L}=1,$
$\norm{\alpha_{1}}=1.$ In order to apply lemma \ref{lem:Let-A-be},
we need $\left\Vert \delta\right\Vert \leq1.$ From formulas \ref{eq:1}
we see that $\left\Vert \delta\right\Vert \leq1$ if $\left|{\displaystyle \frac{1}{q-q^{-1}}}\right|_{L}R_{K}\leq R_{F}.$
So, under this condition, Gauss $R_{E}-$norm is a norm on $\uqr[-,L]b{R_{K},\, R_{F}}\left[E,\alpha_{1},\delta\right]$
and 
\[
\uqr[2,L]{sl}{R_{K},R_{F},R_{E}}:=\uqr[-,L]b{R_{K},\, R_{F}}\left\{ E/R_{E},\alpha_{1},\delta\right\} 
\]
 is a $K-$Banach algebra. 

Note that, due to symmetry between $F$ and $E,$ instead of condition
$\left|{\displaystyle \frac{1}{q-q^{-1}}}\right|_{L}R_{K}\leq R_{F},$
we can take $\left|{\displaystyle \frac{1}{q-q^{-1}}}\right|_{L}R_{K}\leq R_{E}$
(and first extend $\hat{A}_{0}$ by $E$ instead of $F$).

In case $R_{K}=1,$ $\uqr[2,L]{sl}{1,R_{F},R_{E}}$ is a Banach Hopf
algebra ($R_{F}$ or $R_{E}\geq\left|\left(q-q^{-1}\right)^{-1}\right|_{L}$),
isomorphic to $\uqr[2,L]{sl}{R_{F},R_{E}}$.

The projective limit ${\displaystyle \lim_{\from}\uqr[2,L]{sl}{R_{K},R_{F},R_{E}}}$
is a noncommutative Fréchet algebra. One can prove the Fréchet-Stein
property similarly to the previous section, although one has to keep
track of the relation between $R_{K}$ and $R_{F}$ (or $R_{E}$).
When $R_{K}=1$ is fixed, we get the same Fréchet Hopf algebra $\huq{sl}{2,L}$.

\section{Completion of the coordinate algebra $SL_{q}\left(2,L\right)$.}

\subsection{Preliminaries on quantum group $SL_{q}\left(2,L\right)$.}

The quantum matrix algebra $M_{q}(2,L)$ is a bialgebra, defined as
a quotient of free algebra $L\left\langle a,b,c,d\right\rangle $
by the following relations

\[
ab=qba,\mbox{ }ac=qca,\mbox{ }bd=qdb,\mbox{ }cd=qdc,\mbox{ }bc=cb,
\]
\[
ad-da=\left(q-q^{-1}\right)bc\mbox{ .}
\]

The comultiplication is given by formulas

\[
\begin{array}{cc}
\Delta\left(a\right)=a\otimes a+b\otimes c & \Delta\left(b\right)=a\otimes b+b\otimes d\\
\Delta\left(c\right)=c\otimes a+d\otimes c & \Delta\left(d\right)=c\otimes b+d\otimes d
\end{array}\mbox{ .}
\]

The counit is given by formula 
\[
\epsilon\left(\begin{array}{cc}
a & b\\
c & d
\end{array}\right)=\left(\begin{array}{cc}
1 & 0\\
0 & 1
\end{array}\right)\mbox{ .}
\]

The quantum determinant $det_{q}=ad-qbc$ is a central group-like
element in this algebra.

The quantum group $SL_{q}(2,L)$ is the quotient $SL_{q}(2,L)=M_{q}(2,L)$$\diagup\left(det_{q}=1\right)$. 

The set $\left\{ \! a^{n_{a}}b^{n_{b}}c^{n_{c}},\! b^{n_{b}}c^{n_{c}}d^{n_{d}}\,\right\} $
is a vector space basis for $SL_{q}(2,L)$.

$SL_{q}(2,L)$ is a Hopf algebra with the antipode 
\[
S\left(a\right)=d,\mbox{ }S\left(b\right)=-q^{-1}b,\mbox{ }S\left(c\right)=-qc,\mbox{ }S\left(d\right)=a\mbox{ .}
\]

The transposition morphism $\theta_{\alpha,\beta}$ is an automorphism
of $SL_{q}\left(2,L\right)$, given by the following formulas 
\[
\theta_{\alpha,\beta}\left(a\right)=\alpha a,\mbox{ }\theta_{\alpha,\beta}\left(b\right)=\beta c,\mbox{ }\theta_{\alpha,\beta}\left(c\right)=\beta^{-1}b,\mbox{ }\theta_{\alpha,\beta}\left(d\right)=\alpha^{-1}d\mbox{ .}
\]

\subsection{\label{sub:Duality-between-}Duality between $U_{q}\left(\mathfrak{sl}_{2,L}\right)$
and $SL_{q}(2,L)$.}

The algebra $\breve{U}_{q}\left(\mathfrak{sl}_{2,L}\right)$ \cite[3.1.2]{KSch}
has the same generators as $U_{q}\left(\mathfrak{sl}_{2,L}\right)$,
but different relations
\[
\begin{array}{c}
K\cdot K^{-1}=K^{-1}\cdot K=1\mbox{ ,}\\
KE=qEK\mbox{ ,}\\
KF=q^{-1}FK\mbox{ ,}\\
EF-FE={\displaystyle \frac{K^{2}-K^{-2}}{q-q^{-1}}}\mbox{ .}
\end{array}
\]
There is an injective algebra homomorphism 
\[
\begin{array}{c}
\phi:\end{array}\begin{array}{ccc}
U_{q}\left(\mathfrak{sl}_{2,L}\right) & \longrightarrow & \breve{U}_{q}\left(\mathfrak{sl}_{2,L}\right)\\
E & \longmapsto & EK\\
F & \longmapsto & K^{-1}F\\
K & \longmapsto & K^{2}
\end{array}.
\]
These two algebras are not isomorphic. Both algebras admit an automorphism
\[
\begin{array}{c}
\theta_{\alpha}:\end{array}\begin{array}{ccc}
E & \longmapsto & \alpha E\\
F & \longmapsto & \alpha^{-1}F\\
K & \longmapsto & K
\end{array}.
\]

There is a non-degenerate Hopf algebra pairing $\left\langle \cdot,\cdot\right\rangle $
between $U_{q}\left(\mathfrak{sl}_{2,L}\right)$ and $SL_{q}(2,L)$,
see \cite[I.4.4]{KSch}. First, define a pairing $\left\langle \cdot,\cdot\right\rangle \breve{}$
for $\breve{U}_{q}\left(\mathfrak{sl}_{2,L}\right)$ and $SL_{q}(2,L)$,
\cite[I.4.4 Prop.22]{KSch}:

\[
\begin{array}{c}
\left\langle K^{m}E^{n}F^{l},d^{s}c^{r}b^{t}\right\rangle \breve{}=q^{\left(n-r\right)^{2}}{\displaystyle \left[{s\atop n-r}\right]_{q^{2}}}\gamma_{mnt}^{srt}\\
\mbox{if }0\leq n-r=l-t\leq s\mbox{, }\left\langle K^{m}E^{n}F^{l},d^{s}c^{r}b^{t}\right\rangle \breve{}=0\mbox{ otherwise, and}\mbox{ }\\
\left\langle K^{m}E^{n}F^{l},a^{s}c^{r}b^{t}\right\rangle \breve{}=\delta_{rn}\delta_{tl}\gamma_{mnt}^{-srt}\mbox{ ,}
\end{array}
\]
where
\[
\gamma_{mnt}^{srt}=\frac{q^{m\left(s+r-t\right)/2}q^{-s\left(n+l\right)/2}}{q^{n\left(n-1\right)/2}q^{l\left(l-1\right)/2}}\frac{\left(q^{2};q^{2}\right)_{l}\left(q^{2};q^{2}\right)_{n}}{\left(1-q^{2}\right)^{l+n}}\mbox{ ,}
\]
\[
\left(a;q\right)_{n}=\left(1-a\right)\left(1-aq\right)\ldots\left(1-aq^{n-1}\right)\mbox{ .}
\]

The pairing $\left\langle \cdot,\cdot\right\rangle $ can be defined
via identity 
\[
\left\langle x,y\right\rangle =\left\langle \phi\left(x\right),\theta_{1,q^{-1/2}}\left(y\right)\right\rangle \breve{}=\left\langle \left(\phi\circ\theta_{q^{1/2}}\right)\left(x\right),y\right\rangle \breve{}\mbox{ .}
\]

\begin{rem}
\label{rem:Direct-check-shows}Direct check shows that if $\left|1-q\right|_{L}<1$,
then $\left|\gamma_{mnt}^{srt}\right|_{L}=1$ . For $L=\mathbb{Q}_{p}$
the condition $\left|1-q\right|_{L}<1$ corresponds to the case when
$q=\exp(h)$ for some $h\in\mathbb{Z}_{p}$ s.t. $\exp(h)$ converges,
i.e. the case when $SL_{q}(2,L)$ and $U_{q}\left(\mathfrak{sl}_{2,L}\right)$
are deformations of $SL(2,L)$ and $U\left(\mathfrak{sl}_{2,L}\right)$
respectively.
\end{rem}

\subsection{Completion of $SL_{q}(2,L).$ }

The pairing from \ref{sub:Duality-between-} gives the following pairing
between $\uqk{sl}{2,L}$ and $SL_{q}(2,L)$:
\[
\begin{array}{c}
\left\langle K^{m}E^{n}F^{l},d^{s}c^{r}b^{t}\right\rangle =q^{\left(n-r\right)^{2}}{\displaystyle \left[{s\atop n-r}\right]_{q^{2}}}\gamma_{\left(m+n-l\right)nt}^{srt}\cdot q^{-\left(\frac{n\left(n+1\right)}{2}+\frac{l\left(l-1\right)}{2}+ln\right)}\\
\mbox{if }0\leq n-r=l-t\leq s\mbox{, }
\end{array}
\]
\[
\left\langle K^{m}E^{n}F^{l},d^{s}c^{r}b^{t}\right\rangle =0\mbox{ otherwise, and}\mbox{ }
\]
\[
\left\langle K^{m}E^{n}F^{l},a^{s}c^{r}b^{t}\right\rangle =\delta_{rn}\delta_{tl}\gamma_{\left(m+n-l\right)nt}^{-srt}\cdot q^{-\left(\frac{n\left(n+1\right)}{2}+\frac{l\left(l-1\right)}{2}+ln\right)}\mbox{ .}
\]
 This pairing gives a linear inclusion of $SL_{q}(2,L)$ into the
linear dual of $\uqk{sl}{2,L}$. The norms $\nu_{R_{E},R_{F}}'$ (\ref{eq:norm2})
on $\huq{sl}{2,L}$ make $\uqk{sl}{2,L}$ into a normed space. If
elements of $SL_{q}(2,L)$ are continuous (=bounded) w.r.t. $\nu_{R_{E},R_{F}}'$,
then we have an embedding $SL_{q}(2,L)$ into the continuous dual
$\left(\uqr[2,L]{sl}{\nu_{R_{E},R_{F}}'}\right)'_{b}$ of $\uqr[2,L]{sl}{\nu_{R_{E},R_{F}}'}$
($\uqr[2,L]{sl}{\nu_{R_{E},R_{F}}'}$ is the completion of $\huq{sl}{2,L}$
w.r.t. $\nu_{R_{E},R_{F}}'$).

The norm of $\alpha\in SL_{q}(2,L)$ as of a functional on $\uqr[2,L]{sl}{\nu_{R_{E},R_{F}}'}$
is given by identity 
\[
\left\Vert \alpha\right\Vert _{\nu_{R_{E},R_{F}}'}^{*}=\sup_{x}\left|\left\langle x,\alpha\right\rangle \right|_{L},\mbox{ }x\in\uqr[2,L]{sl}{\nu_{R_{E},R_{F}}'}.:\nu_{R_{E},R_{F}}'\left(x\right)\leq1.
\]
Since (from \cite[2.1.1. (3)]{KSch})
\[
\left(q^{2},q^{2}\right)_{m}=\left[m\right]_{q}!\cdot\left(1-q^{2}\right)^{m}\cdot q^{\frac{m\left(m-1\right)}{2}},
\]
then in case $\left|q\right|_{L}=1$ one can check that 
\[
\left|\gamma_{xnt}^{srt}\right|_{L}=\left|\left[n\right]_{q}!\right|_{L}\left|\left[l\right]_{q}!\right|_{L}.
\]
 Thus 
\[
\left|\left\langle K^{m}E^{n}F^{l},a^{s}c^{r}b^{t}\right\rangle \right|_{L}=\delta_{rn}\delta_{tl}\left|\left[n\right]_{q}!\right|_{L}\left|\left[l\right]_{q}!\right|_{L}
\]
 and we have found the first norm 
\[
\norm{a^{s}c^{r}b^{t}}_{\nu_{R_{E},R_{F}}'}^{*}=R_{E}^{-r}R_{F}^{-t}.
\]
For elements of the form $d^{s}c^{r}b^{t}$ we have
\[
\left|\left\langle K^{m}\frac{E^{n}}{\left[n\right]_{q}!}\frac{F^{l}}{\left[l\right]_{q}!},d^{s}c^{r}b^{t}\right\rangle \right|_{L}=\left|\left[{s\atop n-r}\right]_{q^{2}}\right|_{L}\leq1
\]
 due to \cite[6.1.1 (1.8)]{Kas} and \cite[4.1.1.2]{LDiv}. Thus we
have found the second norm 
\[
\norm{d^{s}c^{r}b^{t}}_{\nu_{R_{E},R_{F}}'}^{*}=R_{E}^{-r}R_{F}^{-t}.
\]

Since $\left\{ a^{s}c^{r}b^{t},d^{s}c^{r}b^{t}\right\} $ is a linear
basis of $SL_{q}\left(2,L\right)$, we have an embedding $SL_{q}\left(2,L\right)\hookrightarrow\left(\uqr[2,L]{sl}{\nu_{R_{E},R_{F}}'}\right)'_{b}$.We
denote by $C_{R_{E},R_{F}}^{an}\left(SL_{q}\left(2,L\right)\right)$
the closure of $SL_{q}\left(2,L\right)$ inside $\left(\uqr[2,L]{sl}{\nu_{R_{E},R_{F}}'}\right)'_{b}$,
which consist of series $\sum\alpha_{nmk}a^{n}b^{m}c^{k}+\sum\beta_{mkl}b^{m}c^{k}d^{l}$
with 
\[
\lim_{n,m,k\to\infty}\left|\alpha_{nmk}\right|_{L}\left(\frac{1}{R_{E}}\right)^{k}\left(\frac{1}{R_{F}}\right)^{m}=0,
\]
\[
\lim_{m,k,l\to\infty}\left|\beta_{mkl}\right|_{L}\left(\frac{1}{R_{E}}\right)^{k}\left(\frac{1}{R_{F}}\right)^{m}=0.
\]
It is a Banach algebra w.r.t. supremum-$R_{E}^{-1},R_{F}^{-1}$ norm.
The comultiplication, counit and antipode are also bounded, so it
is a $L-$Banach Hopf algebra.

The injective limit $C^{\omega}\left(SL_{q}\left(2,L\right)\right)={\displaystyle \lim_{\to}C_{R_{E},R_{F}}^{an}\left(SL_{q}\left(2,L\right)\right)}$
is a LCVS of compact type (similar to \cite[16.11]{NFA}). So, $C^{\omega}\left(SL_{q}\left(2,L\right)\right)$
is a noncommutative and noncocommutative $L-$Hopf algebra of compact
type.
\begin{rem}
More generally, one can take $R_{a},$ $R_{b},$ $R_{c},$ $R_{d}<1$
and define a completion of $SL_{q}\left(2,L\right)$ in a similar
way. The norm will not be submultiplicative (i.e. multiplication is
not continuous w.r.t. supremum-norm), but comultiplication will be
bounded, so in this case we get a noncocommutative coalgebra. More
over, it is still can be done if $\left|q\right|_{L}<1$ and in this
case the Haar functional of $SL_{q}\left(2,L\right)$ \cite[4.2.6]{KSch}
is bounded.\end{rem}

\end{document}